\documentclass[a4paper,11pt]{amsart}

\usepackage[utf8]{inputenc}
\usepackage[T1]{fontenc}

\usepackage[margin=1.2in]{geometry}

\usepackage{amsmath}
\usepackage{amssymb}
\usepackage{amsthm}

\usepackage{eucal} 

\usepackage{tikz-cd} 

\usepackage{enumitem}

\theoremstyle{plain}
\newtheorem{thm}{Theorem}[section]
\newtheorem{prop}[thm]{Proposition}
\newtheorem{cor}[thm]{Corollary}
\newtheorem{lemma}[thm]{Lemma}

\theoremstyle{definition}
\newtheorem{dfn}[thm]{Definition}
\newtheorem{rmk}[thm]{Remark}
\newtheorem{ex}[thm]{Example}

\newcommand{\Acal}{\mathcal{A}}

\newcommand{\Ccal}{\mathcal{C}}
\newcommand{\Dcal}{\mathcal{D}}
\newcommand{\Ecal}{\mathcal{E}}
\newcommand{\Fcal}{\mathcal{F}}
\newcommand{\Gcal}{\mathcal{G}}

\newcommand{\Ocal}{\mathcal{O}}
\newcommand{\Pcal}{\mathcal{P}}

\newcommand{\Rcal}{\mathcal{R}}
\newcommand{\Scal}{\mathcal{S}}
\newcommand{\Tcal}{\mathcal{T}}

\newcommand{\Vcal}{\mathcal{V}}

\newcommand{\Xcal}{\mathcal{X}}

\newcommand{\Zcal}{\mathcal{Z}}
\newcommand{\Abb}{\mathbb{A}}
\newcommand{\Cbb}{\mathbb{C}}

\newcommand{\Hbb}{\mathbb{H}}

\newcommand{\Pbb}{\mathbb{P}}

\newcommand{\Zbb}{\mathbb{Z}}
\newcommand{\tbf}{\mathbf{t}}
\renewcommand{\phi}{\varphi}

\DeclareMathOperator{\Hom}{\mathsf{Hom}}
\DeclareMathOperator{\End}{\mathsf{End}}

\DeclareMathOperator{\im}{im}
\DeclareMathOperator{\Spec}{\mathsf{Spec}}
\DeclareMathOperator{\Supp}{\mathsf{Supp}}

\newcommand{\mf}{\mathfrak{m}}
\newcommand{\pf}{\mathfrak{p}}
\newcommand{\qf}{\mathfrak{q}}

\newcommand{\newterm}[1]{\textbf{#1}}

\let\mod\relax
\DeclareMathOperator{\Mod}{\mathsf{Mod}}
\DeclareMathOperator{\Qcoh}{\mathsf{Qcoh}}
\DeclareMathOperator{\coh}{\mathsf{coh}}
\DeclareMathOperator{\Proj}{\mathsf{Proj}}
\DeclareMathOperator{\Inj}{\mathsf{Inj}}
\DeclareMathOperator{\Add}{\mathsf{Add}}

\DeclareMathOperator{\mod}{\mathsf{mod}}
\DeclareMathOperator{\fp}{\mathsf{fp}}

\DeclareMathOperator{\tors}{\mathsf{tors}}
\DeclareMathOperator{\len}{\mathsf{len}}
\DeclareMathOperator{\serre}{\mathsf{filt}}
\DeclareMathOperator{\Ass}{\mathsf{Ass}}
\DeclareMathOperator{\Ann}{\mathsf{Ann}}
\DeclareMathOperator{\mcrk}{\mathsf{rk}}
\DeclareMathOperator{\gldim}{\mathsf{gl.dim}}
\DeclareMathOperator{\idim}{\mathsf{i.dim}}

\DeclareMathOperator{\D}{\mathsf{D}}

\begin{document}

\title{Torsion-simple objects in abelian categories}
\author{Sergio Pavon}
\address[Sergio Pavon]{Dipartimento di Matematica ``Tullio Levi-Civita'',
Università di Padova, Via Trieste 63, 35121 Padova (PD), Italy. }
\email{sergio.pavon@math.unipd.it}

\thanks{This work was supported by the Department of Mathematics of
the University of Padova via its BIRD~2022, through the project \emph{Representations of quivers with
commutative coefficients}.}

\begin{abstract}
	We introduce the notion of torsion-simple objects in an abelian category:
	these are the objects which are always either torsion or torsion-free with
	respect to any torsion pair. We present some general results concerning their
	properties, and then proceed to investigate the notion in various contexts,
	such as the category of modules over an artin algebra or a commutative
	noetherian ring, and the category of quasi-coherent sheaves over the
	projective line.
\end{abstract}
\maketitle

\section*{Introduction}

In an abelian category, an important tool is the theory of torsion pairs, whose
main feature is to provide decompositions: inside every object, a torsion
pair `cuts out' a torsion subobject, its \emph{torsion part}. In practice,
torsion pairs are produced in many ways: from combinatorial data of some
kind \cite{pavo-vito-21,sten-75}, from (co)tilting and (co)silting
theory~\cite{adac-iyam-reit-14,ange-hrbe-17,ange-mark-vito-16a,colp-trli-95},
from $t$-structures via HRS-tilting \cite{happ-reit-smal-96}, from stability
conditions \cite{tatt-tref-23}, just to mention a few. Often, one obtains a
certain torsion pair (or a collection of them) through one of these means, and
then uses it to decompose each object of the category.

In this paper we reverse
this perspective: we fix an object $x$ and let the torsion pair vary, obtaining the
collection of all the possible torsion parts of $x$. Our focus is on the objects which,
despite possibly having many subobjects, only have trivial torsion parts: that
is, objects that are always either torsion or torsion-free, with respect to any torsion
pair. We call these objects \emph{torsion-simple},
as even if at times they can be rather complex (we
provide examples in the paper, see in particular \S\ref{subsec:commutative}),
they look \emph{simple} (in the technical sense of not having non-trivial
subobjects) from the point of view of torsion theory.

Knowing examples, or even a classification, of such objects could prove to be
useful in the concrete situations outlined above. Indeed, independently of how a
torsion pair is produced and described, the torsion-simple objects will neatly
divide into two groups, one falling in the torsion class, the other in the
torsion-free class.

\subsection*{Structure of the paper}
The paper is organised in the following way.

In \textbf{Section~1} we recall some preliminaries on torsion pairs and Grothendieck
categories.

In \textbf{Section~2} we give the definition of torsion-simple objects, and then
we establish some general facts. The results in this section have two
flavours: some allow us to test an objects for torsion-simplicity in a more
favourable setting, where there are fewer torsion pairs; the others reduce
\emph{a priori} the set of subobjects of an object which could be torsion parts.

In \textbf{Section~3} we apply the tools developed in the previous section to study
torsion-simple objects in a few concrete settings: categories of
finite length and their Grothendieck counterparts (this includes $\mod A$ and
$\Mod A$ for an artinian ring $A$), in \S\ref{subsec:finite-length}; the categories
$\mod R$ and $\Mod R$ of modules over a
commutative noetherian ring $R$, in \S\ref{subsec:commutative}; and finally the
categories of coherent and quasi-coherent sheaves over the projective line, in
\S\ref{subsec:sheaves}.

In the first two situations, we obtain a complete characterisation
of torsion-simple finitely generated objects. In a locally finite Grothendieck
category (such as $\Mod A$ for an artinian ring $A$) they are the finitely
generated objects with a
unique simple factor (Proposition~\ref{prop:finite-length}). Over a commutative
noetherian ring, they are the finitely generated modules with a unique associated prime
(Corollary~\ref{cor:ass-singleton}).

For infinitely generated objects, we present several examples (such as the
residue fields of a commutative ring, see Example~\ref{ex:residues}) and counterexamples. In the
commutative setting, we have more insight when $R$ is a Dedekind domain (see
Proposition~\ref{prop:dedekind}, which shows that the indecomposable injectives
are torsion-simple).

On the way, we also prove that for a ring to be torsion-simple as a module over
itself is a Morita-invariant property, generalising a result of
\cite{bica-jamb-kepk-neme-73}; see Theorem~\ref{thm:morita}.

Section~3 ends with a subsection containing a more elementary proof, based on
commutative algebra, of the fact that the cyclic module $R/\pf$ is
torsion-simple in $\Mod R$, for every prime $\pf$ of a commutative noetherian ring $R$.

\subsection*{Acknowledgments}
We would like to thank Giovanna Le Gros and Jorge Vitória for useful discussions
on the topic of this paper.

\section{Preliminaries}

\subsection{Notation and conventions}\label{subsec:conventions}
Let $\Acal$ be an abelian category. We will use lowercase letters $w,x,\dots$ to denote
the objects of $\Acal$, and the symbol $\Acal(-,-)$ for the $\Hom$-bifunctor.
For a class $\Scal\subseteq\Acal$ of objects, we will consider the two
$\Hom$-orthogonals
\[
	\Scal^\bot:=\{\,x\in\Acal\mid \Acal(s,x)=0\;\;\forall s\in\Scal\,\}
	\quad\text{and}\quad
	{}^\bot\Scal:=\{\,x\in\Acal\mid \Acal(x,s)=0\;\;\forall s\in\Scal\,\}.
\]
For a pair $(\Xcal,\Zcal)$ of classes of objects of $\Acal$, its \emph{span},
the class of extensions of objects of $\Zcal$ by objects of $\Xcal$, will be
denoted by 
\[\Xcal\ast\Zcal:=\{\,y\in\Acal\mid \exists (0\to x\to y\to z\to 0)\text{ exact with
}x\in\Xcal, z\in\Zcal\,\}.\]

In this paper we assume all our categories to be well-powered: that is, the
subobjects of any given object form a set, as opposed to a proper class. This is
the case for (subcategories of) categories with a (co)generator
\cite[Prop.~3.35]{frey-03}, and for skeletally small categories. We also assume
all our subcategories to be strict and full.

For a ring $R$, we will denote by $\Mod R$ the category of right $R$-modules,
and by $\mod R\subseteq\Mod R$ the (skeletally small) subcategory of
finitely presented $R$-modules.

\subsection{Finiteness conditions on abelian categories}
To state some of our results and examples, we will impose several finiteness
conditions on our abelian categories; we now review the necessary definitions.

First, let $\Gcal$ be a Grothendieck category. Recall that
an object $x\in\Gcal$ is \emph{finitely
presented} if $\Gcal(x,-)$ commutes with direct limits; we denote by $\fp\Gcal$
the subcategory of finitely presented objects. We say that $\Gcal$ is \emph{locally
finitely presented} if $\fp\Gcal$ is skeletally small and its closure under
direct limits is $\varinjlim\fp\Gcal=\Gcal$; see for example \cite{craw-94}. For
instance, $\Mod R$ is always locally finitely presented, with $\fp(\Mod R)=\mod
R$.

We will be particularly interested in the case in which $\fp\Gcal$ is itself
abelian: when this happens, $\Gcal$ is said to be \emph{locally coherent}. This is the case
for $\Mod R$ when $R$ is a right-coherent ring. There is a bijection between
(equivalence classes of) skeletally small abelian categories $\Acal$ and locally
coherent Grothendieck categories $\Gcal$, where $\Acal=\fp\Gcal$ and $\Gcal$ is
constructed as a ``direct limit completion'' of $\Acal$ (see \cite[\S1.4]{craw-94}).

We say that an object of an abelian category $\Acal$ is \emph{noetherian}
(resp.\ \emph{artinian}) if its poset of subobjects satisfies the ACC (resp.\
the DCC); it is \emph{of finite length} if it is both noetherian and artinian,
or equivalentely if it admits a finite filtration with simple factors.
The category $\Acal$ is then \emph{noetherian} (resp.\ \emph{artinian}, \emph{of
finite length}) if so are all its objects.

A locally coherent Grothendieck category $\Gcal$ is \emph{locally noetherian}
(resp. \emph{locally coperfect}~\cite{roos-69}, \emph{locally finite}) if $\fp\Gcal$ is
noetherian (resp. artinian, of finite length). For example, $\Mod R$ is locally
noetherian if $R$ is right-noetherian, locally coperfect if $R$ is
right-coherent and \emph{left}-perfect~(see again \cite[p.~206]{roos-69}),
locally finite if $R$ is right-artinian (hence also right-noetherian).

\subsection{Torsion pairs} Let $\Acal$ be an abelian category. Recall the
following definition:

\begin{dfn}[Dickson~\cite{dick-66}]
	A pair of classes $(\Tcal,\Fcal)$ of objects of $\Acal$ is a
	\newterm{torsion pair} if:
	\begin{enumerate}[label=(T\arabic*)]
		\item $\Acal(\Tcal,\Fcal)=0$ (they are orthogonal);
		\item $\Acal=\Tcal\ast\Fcal$ (they span $\Acal$).
	\end{enumerate}
	In this case, the class $\Tcal$ (resp.\ $\Fcal$) and its objects are called
	\newterm{torsion} (resp.\ \newterm{torsion-free}) with respect to the pair.
\end{dfn}

From the interplay between (T1) and (T2), one easily draws the following
consequences:
\begin{enumerate}[label=(\roman*)]
	\item $\Tcal={}^\bot\Fcal$ and $\Tcal^\bot=\Fcal$, so each class determines
		the other.
	\item For any object $x\in \Acal$, the sequence $(0\to tx\to x\to fx\to 0)$
		with $tx\in\Tcal, fx\in\Fcal$ provided by (T2) is unique up to
		isomorphism; it is called the \newterm{torsion sequence} of $x$ with
		respect to the torsion pair. This yields functors $t\colon\Acal\to
		\Tcal\subseteq\Acal$, called the \newterm{torsion radical}, and
		$f\colon\Acal\to\Fcal\subseteq\Acal$, called the \newterm{torsion-free
		coradical}, which are the right- and the left-adjoint of the inclusions
		$\Tcal\subseteq\Acal$ and $\Fcal\subseteq\Acal$, respectively. 
		The object $tx$ (resp.\ $fx$) is called the \newterm{torsion} (resp.
		\newterm{torsion-free}) \newterm{part} of $x$ with respect to the pair.
\end{enumerate}
As an immediate consequences of (i), $\Tcal$ is closed under extensions,
quotients and existing colimits (in particular, existing coproducts); $\Fcal$ is
closed under extensions, subobjects and existing limits (in particular, existing
products). We say that the torsion pair is \newterm{hereditary} if $\Tcal$ is
also closed under subobjects, and \newterm{of finite type} if $\Fcal$ is closed
under (existing) direct limits.

Under some additional assumptions on $\Acal$, these closure
properties are enough to characterise torsion and torsion-free classes, as
we now recall.

\begin{lemma}
	Let $\Tcal\subseteq\Acal$ be closed under extensions and quotients. Then
	$(\Tcal,\Tcal^\bot)$ is a torsion pair if and only if 
	$\Acal$ satisfies the following \emph{Largest Subobject
	Condition} with respect to $\Tcal$:
	\begin{enumerate}
		\item[(LSC)] any object of $\Acal$ admits a largest subobject belonging
			to $\Tcal$.
	\end{enumerate}
\end{lemma}

\begin{proof}
	The implication ($\Rightarrow$) follows immediately by our previous
	observation that the torsion radical is right-adjoint to the inclusion of
	the torsion class.

	For the converse, $(\Tcal,\Tcal^\bot)$ obvioulsy verifies (T1). To verify
	(T2), for any object $x\in\Acal$ we let $w\subseteq x$ be its largest
	subobject belonging to $\Tcal$, and show that $x/w\in\Tcal^\bot$, so that
	the sequence $(0\to w\to x\to x/w\to 0)$ has the prescribed form.

	Notice that since $\Tcal$ is closed under quotients, by considering
	images of morphisms we may reduce ourselves to prove that no non-zero
	subobject of $x/w$ belongs to $\Tcal$. Assume by contradiction that
	$0\neq y\subseteq x/w$ belongs to $\Tcal$: then its ``preimage under
	$x\twoheadrightarrow x/w$'' (that is, the pullback of $y\hookrightarrow
	x/w\twoheadleftarrow x$) is a subobject of $x$ which belongs to $\Tcal$, since it is an
	extension of $y$ by $w$; but it is not contained in $w$, a contradiction.
\end{proof}

Applying the lemma to $\Acal^{op}$, one obtains a
characterisation of torsion-free classes in $\Acal$.

These characterisations are useful in view of the fact that under some
assumptions on $\Acal$, the (LSC) (or its dual) follows from closure properties.
Recall that by \S\ref{subsec:conventions} we are assuming $\Acal$ to be
well-powered, so the subobjects of any given object form a set.

\begin{lemma}
	$\Acal$ satisfies the (LSC) with respect to any additive subcategory closed under
	extensions, quotients and existing coproducts as soon as either:
	\begin{enumerate}
		\item $\Acal$ is noetherian; or
		\item $\Acal$ has arbitrary coproducts.
	\end{enumerate}
\end{lemma}

\begin{proof}
	Let $\Tcal\subseteq\Acal$ have the closure properties in the statement,
	and let $x\in\Acal$ be any object. Consider the poset of subobjects of $x$
	belonging to $\Tcal$: it is non-empty, as it contains $0$, and any ascending
	chain in it has an upper bound. Indeed, if $\Acal$ is noetherian, a chain
	must stabilise; if $\Acal$ admits arbitrary coproducts, the union of a
	chain, being a quotient of the coproduct, still belongs to
	$\Tcal$ (by `union', we mean the image of the canonical morphism from the
	direct limit of the chain to $x$, which contains all the terms of the chain). By
	Zorn's Lemma, there exists a maximal element in this poset. To show that it
	is in fact a maximum, just observe that the sum of two subobjects of $x$
	belonging to $\Tcal$, being a quotient of their (finite) coproduct, still
	belongs to $\Tcal$. This maximum is then the largest subobject of $x$
	belonging to $\Tcal$.
\end{proof}

We spell out the consequence and its dual:

\begin{prop}\label{prop:1-4}
	Let $\Acal$ be an abelian category. If either:
	\begin{enumerate}
		\item $\Acal$ is noetherian; or
		\item $\Acal$ has arbitrary coproducts,
	\end{enumerate}
	then $\Tcal\subseteq\Acal$ is a torsion class if and only if it is closed
	under extensions, quotients and existing coproducts. Dually, if either:
	\begin{enumerate}
		\setcounter{enumi}{2}
		\item $\Acal$ is artinian; or
		\item $\Acal$ has arbitrary products,
	\end{enumerate}
	then $\Fcal\subseteq\Acal$ is a torsion-free class if and only if it is
	closed under extensions, subobjects and existing products.
\end{prop}

This discussion was motivated by the following construction, which we will need.

Let $\Scal\subseteq\Acal$ be a class of objects. The classes
${}^\bot(\Scal^\bot)$ and ${}^\bot\Scal$ are closed under extensions, quotients
and existing coproducts, while the classes $\Scal^\bot$ and
$({}^\bot\Scal)^\bot$ are closed under extensions, subobjects and existing
products. Therefore, under any of the hypotheses (1--4) of the previous
Proposition, we can construct the torsion pairs:
\[
	({}^\bot(\Scal^\bot),\Scal^\bot)\;\;\text{\newterm{generated} by
	}\Scal,\text{ and }\quad
	({}^\bot\Scal,({}^\bot\Scal)^\bot)\;\;\text{\newterm{cogenerated} by }\Scal
\]
(notice that indeed $({}^\bot(\Scal^\bot))^\bot=\Scal^\bot$, and similarly on the other
side). As the terminology suggests, one easily sees that the torsion class
generated by $\Scal$ (resp.\ the torsion-free class cogenerated by $\Scal$) is
the smallest containing $\Scal$.

\section{Torsion-simple objects}

Let $\Acal$ be an abelian category; from now on, we assume that:
\begin{equation}\label{eq:hyp-star}\tag{$\star$}
	\text{the torsion pairs generated and cogenerated by a single object exist.}
\end{equation}
In our examples, we will usually be under one of the hypotheses (1--4) of
Proposition~\ref{prop:1-4}.

For any object $x$ of $\Acal$, denote its set of torsion parts by:
\[\tbf_\Acal x:=\{\,w\subseteq x\mid \exists (\Tcal,\Fcal)\text{ torsion
pair in }\Acal\text{ with }w\in\Tcal, x/w\in\Fcal\,\}.\]
When $\Acal$ is clear from context, we omit it from the notation, writing $\tbf
x:=\tbf_\Acal x$. Notice that by considering the \emph{trivial torsion pairs}
$(0,\Acal), (\Acal,0)$ we have $\{0,x\}\subseteq\tbf x$. The main concern of
this paper are the following objects:

\begin{dfn}
	A non-zero object $x$ of $\Acal$ is \newterm{torsion-simple} if $\tbf
	x=\{0,x\}$. Equivalently, if it is either torsion or torsion-free with
	respect to every torsion pair $(\Tcal,\Fcal)$ (which we can write:
	$x\in\Tcal\cup\Fcal$ for every torsion pair $(\Tcal,\Fcal)$).
\end{dfn}

This definition involves all torsion pairs of $\Acal$; since we assume to be
able to generate torsion pairs, we have the following `intrinsic' criterion to
check whether $x$ is torsion-simple.

\begin{lemma}\label{lemma:torsion-characterisation}
	Under the assumption~\eqref{eq:hyp-star}, for every object $x$ of $\Acal$ we have:
	\[\tbf x=\{\, w\subseteq x \mid \Acal(w,x/w)=0\,\}.\]
\end{lemma}

\begin{proof}
	The inclusion ($\subseteq$) is obvious from (T1). For the converse, any
	$w\subseteq x$ such that $\Acal(w,x/w)=0$ is the torsion part of $x$
	with respect to the torsion pair generated by $w$.
\end{proof}

\begin{cor}
	An object $x$ of $\Acal$ is torsion-simple if and only if for every
	$w\subseteq x$ we have:
	\[\Acal(w,x/w)=0\;\;\Rightarrow\;\; w=0\text{ or }w=x.\]
\end{cor}

Before considering concrete examples, we present a few results that will aid 
our study.

\subsection{Constructing torsion-simple objects}

The following is a way to construct new torsion-simple objects from
known ones.
Recall that the class\footnote{\label{foot:stanley}The same argument used by Stanley
\cite[\S8]{stan-10} to prove that there is proper class of $t$-structures in the
derived category $\mathsf{D}(\Zbb)$ also works to prove that there is a proper
class of torsion pairs in $\Mod \Zbb$.} $\tors\Acal$ of torsion pairs
of $\Acal$ can be partially ordered by inclusion of the torsion classes. Any
torsion-simple object $x$ then induces a partition of $\tors\Acal$ into an
upper-class, consisting of the torsion pairs for which $x$ is torsion, and a
complementary lower-class, consisting of the torsion pairs for which $x$ is
torsion-free.

\begin{dfn}
	The \newterm{type} of a torsion-simple object $x$ of $\Acal$ is the
	subclass of $\tors\Acal$ consisting of torsion pairs for which $x$ is
	torsion.
\end{dfn}

\begin{prop}\label{prop:type}
	The class of torsion-simple objects of a given type is closed under
	extensions and images of morphisms; if $\Acal$ is (AB5), it is also closed under
	coproducts.
\end{prop}

\begin{proof}
	Closure under extensions follows from the fact that both torsion and
	torsion-free classes enjoy the same property. Similarly, if $\Acal$ is
	(AB5), coproducts are subobjects of the corresponding products; hence
	torsion-free classes are also closed under coproducts, and one argues in the
	same way. The image of a morphism bewteen torsion-simple objects of the same
	type is at the same time a quotient of the domain and a subobject of the
	codomain, so it is torsion (resp.\ torsion-free) when both the domain and codomain
	are.
\end{proof}

In \S\ref{subsec:commutative} we will see a special case in which these classes are also
closed under subobjects, because all torsion pairs are hereditary.

\subsection{Testing for torsion-simplicity}

In this subsection we collect some results which help in testing whether an
object $x$ of $\Acal$ is torsion-simple. We have two themes:
showing that we can carry the test inside certain
subcategories containing $x$, rather than in the whole $\Acal$;
and trying to reduce the set of subobjects of $x$ which can be torsion parts, by
finding conditions a torsion part must satisfy.
We start with an instance of the first.

Recall that a \emph{Serre subcategory} of an abelian category $\Acal$ is a subcategory
$\Scal\subseteq\Acal$ closed under subobjects, extensions and quotients. Such
$\Scal$ is itself abelian, since it is closed under kernels and cokernels.

\begin{prop}\label{prop:serre-restriction}
	Let $\Acal$ be an abelian category, and let $\Scal\subseteq\Acal$ be a Serre
	subcategory. For $x\in\Scal$, we have that:
	\begin{enumerate}
		\item $\tbf_\Acal x\subseteq\tbf_\Scal x$; in particular, if $x$ is
			torsion-simple in $\Scal$, it is torsion-simple in $\Acal$;
		\item if $\Acal$ satisfies~\eqref{eq:hyp-star}, $\tbf_\Acal x=\tbf_\Scal x$,
			and we have the converse of the implication in (1).
	\end{enumerate}
\end{prop}

\begin{proof}
	(1) It is enough to observe that since $\Scal$ is closed under subobjects
	and quotients, any torsion pair $(\Tcal,\Fcal)$ in $\Acal$ restricts to a
	torsion pair $(\Tcal\cap\Scal,\Fcal\cap\Scal)$ in $\Scal$, whose torsion
	sequences are those of $(\Tcal,\Fcal)$ for the objects of $\Scal$.

	(2) Assuming that $\Acal$ satisfies~\eqref{eq:hyp-star}, using
	Lemma~\ref{lemma:torsion-characterisation} we have that
	\[\tbf_\Scal x\subseteq\{w\subseteq x\text{ in }\Scal\mid
	\Scal(w,x/w)=0\}\subseteq\{w\subseteq x\text{ in }\Acal\mid \Acal(w,x/w)=0\}=\tbf_\Acal x.\qedhere\]
\end{proof}

\begin{cor}\label{cor:restriction}
	Let $\Gcal$ be a locally noetherian Grothendieck category. An object
	$x\in\fp\Gcal$ is torsion-simple in $\Gcal$ if and only if it is
	torsion-simple in $\fp\Gcal$.
\end{cor}

\begin{proof}
	Follows from the Proposition, as $\fp\Gcal\subseteq\Gcal$ is a Serre
	subcategory by the noetherianity assumption, and $\Gcal$
	satisfies~\eqref{eq:hyp-star} since it has arbitrary (co)products.
\end{proof}

Now, we state the first and most important condition a torsion part always
satisfies; despite being easily observed, it is extremely useful.

\begin{prop}\label{prop:stable}
	Let $\Acal$ be an abelian category, $x$ an object of $\Acal$ and $w$ its
	torsion part with respect to a torsion pair. Then $w$ is \emph{stable under
	endomorphisms}, meaning that any endomorphism of $x$ restricts to an
	endomorphism of $w$:
	\[\im(w\hookrightarrow x\overset{g}\to x)\subseteq w \quad\forall g\colon x\to x.\]
\end{prop}

\begin{proof}
	This follows immediately from the fact the the composition
	\[w\hookrightarrow x\overset{g}\to x\twoheadrightarrow x/w\]
	must vanish for every $g\colon x\to x$ by (T1). Notice that the restriction of $g$
	to $w\to w$ is nothing other then $t(g)$, where $t$ is the torsion radical.
\end{proof}

This fact and its use in characterising torsion parts are not new; for example,
Jans uses it to make the following observation:

\begin{cor}[{\cite[p.~4]{jans-65}}]
	Every torsion part of a ring $R$ in $\Mod R$ is a two-sided ideal.
\end{cor}

\noindent In our terminology, this implies that a simple ring is torsion-simple in its
category of right- (or left-)modules. We will draw a couple more consequences
from this observation.

\begin{cor}\label{cor:division-ring}
	For a ring $R$ and a module $M$ in $\Mod R$, let $S:=\End_R(M)$. If $M$ is
	simple as a $S$-$R$-bimodule, it is torsion-simple in $\Mod R$. In particular, if $R$ is
	commutative, any $R$-algebra which is a division ring is torsion-simple in
	$\Mod R$.
\end{cor}

\begin{proof}
	Using the canonical structure of left $S$-module of $M\in\Mod R$,
	Proposition~\ref{prop:stable} translates into the fact that a torsion part
	of $M$ is a not only a right $R$-submodule, but also a left $S$-submodule;
	the conclusion is immediate. For the second part, if $D$ is an $R$-algebra
	we have $D\simeq\End_R(D)$; and a division ring is simple as a
	module over itself.
\end{proof}

In particular, this applies to the residue fields at the prime ideals of a commutative
ring, as we will do in Example~\ref{ex:residues}.

The second consequence pertains to the other reduction theme.

\begin{cor}\label{cor:localisation}
	Let $R,S$ be rings, with $S$ commutative.
	Consider a ring epimorphism $\phi\colon R\to S$,
	so that $\Mod S\subseteq\Mod R$ as a full subcategory.
	Then, an $S$-module $M$
	is torsion-simple in $\Mod R$ if and only if it is torsion-simple in
	$\Mod S$.
\end{cor}

\begin{proof}
	Let $M$ be an $S$-module. Since $S$ is commutative, we have a
	natural morphism $S\to \End_R(M)$ given by $s\mapsto (-\cdot s)$. Then, an
	$R$-submodule of $M$ which is stable under endomorphism is automatically an
	$S$-submodule. Using the Proposition and the characterisation of
	Lemma~\ref{lemma:torsion-characterisation}, this shows that
	$\tbf_RM=\tbf_SM$, and we conclude immediately.
\end{proof}

We remark that, for this proof to work, the commutativity assumption on $S$ could
be substituted by the apparently weaker requirement that $\phi$ factors through
che center $Z(S)$ of $S$. However, this makes the inclusion $Z(S)\subseteq
S$ a ring epimorphism, forcing $S$ to be commutative, see
\cite[Prop.~1.3(b)]{stor-73}.

This commutativity requirement is really necessary to the Corollary, as shown by
the following counterexample. For more information on the ring epimorphisms from
an artin algebra, see \cite[Thm.~1.6.1]{iyam-03}.

\begin{ex}
	Consider the matrix ring $M_2(k)$, and let $\Lambda\subseteq M_2(k)$ be the
	ring of upper triangular matrices (which is the path algebra of the quiver
	$\Abb_2=(\bullet\to\bullet)$). This inclusion is a ring epimorphism, which does
	not satisfy the assumption of our corollary, as its target is not
	commutative.
	The category $\mod{M_2(k)}\simeq \mod k$ is identified with the
	additive closure $\Ccal\subseteq\mod \Lambda$ of the projective-injective
	indecomposable $\Lambda$-module $P_1$.
	In \S\ref{subsec:finite-length} we will classify torsion-simple objects of
	categories of finite length, such as $\Ccal$ and $\mod\Lambda$. As a result,
	we will see that $P_1$ is torsion-simple in $\Ccal$ (where it is actually simple)
	but not in $\mod \Lambda$ (where it has two distinct simple factors).
\end{ex}

We will apply Corollary~\ref{cor:localisation} in the study of
the injective envelopes of the cyclic modules $R/\pf$, where $R$ is a
commutative ring and $\pf$ a prime: these indeed carry a natural $R_\pf$-module
structure. See \S\ref{subsec:commutative}.

The next result is also in the vein of reducing the set of subobjects to test in
order to identify possible torsion parts. By $\len x$ we denote the
\emph{length} of an object $x$ of finite length: that is, the length of any
filtration of $x$ with simple factors.
\begin{prop}\label{prop:injective}
	Let $\Gcal$ be a locally finitely presented Grothendieck category. Let $E$
	be an object which is the direct union of its subobjects of finite length,
	and let $f\colon E\to E$ be an endomorphism such that:
	\begin{enumerate}
		\item $\ker f$ is essential in $E$, and
		\item $\ker f\subseteq \im f$.
	\end{enumerate}
	Then there is no proper torsion part $T\subsetneq E$ such that $\ker f\subseteq
	T\subseteq \im f$.
\end{prop}

\begin{proof}
	Let $T\subsetneq E$ be such that $\ker f\subseteq T\subseteq \im f$.
	We start by showing that $T\nsubseteq f(T)$.
	Let $0\neq F\subseteq E$ be a
	subobject of finite length $\len(F)$. By (1), we have $F\cap\ker f\neq 0$.
	This means that $f(F)$ is a proper quotient of $F$, and
	therefore $\len(f(F))<\len(F)$. Now, since we
	assume $T\subsetneq E$, it cannot contain all the subobjects of $E$ of
	finite length, whose union is $E$; and since the length is a natural number, we can choose
	$F\neq 0$ to be of minimal length with respect to the property that
	$F\nsubseteq T$. By this choice, it follows that $f(F)\subseteq T$, since it
	has smaller length by our argument so far. Now, assume we had that
	$f(F)\subseteq f(T)$: this translates to $F+\ker f\subseteq T+\ker f=T$, and
	therefore in particular $F\subseteq T$, which contradicts our choice of $F$.
	We conclude as wanted that $T\nsubseteq f(T)$, as the first contains $f(F)$ but the
	latter does not.

	Now, consider the canonical isomorphism $\im f\simeq E/\ker f$, and its
	composition with the projection $E/\ker f\to E/T$. Since the kernel of the
	projection is $T/\ker f$, it can be easily checked
	that the kernel of the composition is $f(T)$. As such, it does not contain
	$T$, so the composition $T\subseteq \im f\simeq E/\ker f\to E/T$ is
	non-zero. We deduce that $T$ cannot be a torsion part of $E$.
\end{proof}

Examples of objects $E$ to which this proposition applies are the objects of a
locally finite Grothendieck category having a simple socle, and the injective
envelope of the residue field of a commutative local ring
\cite[\S31]{lam-99}: see Lemma~\ref{lemma:PID} later on.

\section{Examples}

In this section we investigate torsion-simple objects in various contexts,
giving examples and some interesting counterexamples.

\subsection{Torsion-simple rings}
We start by recalling an older paper of Bican, Jambor, Kepka and
N\v{e}mec, in which they study rings which are torsion-simple as
modules over themselves, under the name
of \emph{$R$-rings}. We review two of their results:

\begin{prop}[{\cite[Prop.~1.13]{bica-jamb-kepk-neme-73}}]
	Assume a ring $R$ has no non-trivial idempotent two-sided ideals. Then a
	non-trivial projective module is never torsion for any non-trivial torsion pair.
\end{prop}

Under the hypothesis of this proposition, we deduce that a non-zero projective module
is torsion-simple if and only if it is torsion-free for every non-trivial
torsion pair.

The following is one of the main results of the cited paper:

\begin{thm}[{\cite[Thm.~2.8]{bica-jamb-kepk-neme-73}}]
	For a ring $R$, denote by $M_n(R)$ the ring of $n\times n$ matrices over
	$R$. The following are equivalent in $\Mod R\simeq \Mod M_n(R)$:
	\begin{enumerate}[label=(\roman*)]
		\item $M_n(R)$ is torsion-simple for some $n\geq 1$;
		\item $M_n(R)$ is torsion-simple for every $n\geq 1$.
	\end{enumerate}
\end{thm}

\noindent We are able to vastly generalise this theorem:

\begin{thm}\label{thm:morita}
	Being torsion-simple, for a ring, is a Morita-invariant property.
\end{thm}

\begin{proof}
	By Proposition~\ref{prop:type}, in a Grothendieck category the class of
	torsion-simple objects of a given type is closed under coproducts and direct
	summands (which are images of endomorphisms). Therefore, a ring $R$ is
	torsion-simple in $\Mod R$ if and only if every projective $R$-module is
	torsion-simple, as $\Proj R=\Add R$. This is a categorical property, and
	therefore it carries over along Morita equivalence.
\end{proof}

\subsection{Torsion-simple objects filtered by simples}\label{subsec:finite-length}

We now focus on objects which admit a filtration with simple factors: at first
finite, and later infinite. This will in particular be a complete result for
categories of finite length, such as $\mod A$ for an artinian ring $A$, and
their locally finite Grothendieck completion, such as $\Mod A$.

We apologise to the reader for the first example, which we had to write, sooner
or later:
\begin{ex}
	Simple objects are torsion-simple.
\end{ex}

\begin{prop}\label{prop:finite-length}
	Let $\Acal$ be a noetherian category. An object of finite length $x$ of
	$\Acal$ is torsion-simple if and only if it has a unique simple factor, up
	to isomorphism.
\end{prop}

\begin{proof}
	The implication $(\Leftarrow)$ follows from the fact that simple objects are
	torsion-simple combined with Proposition~\ref{prop:type}, as the class of
	torsion-simple objects of a given type is closed under extensions. For the
	converse, let $s$ be any simple object, and consider the class $\serre(s)$ of
	objects admitting a finite filtration with factors isomorphic to $s$. It is
	closed under extensions and quotients, hence, assuming $\Acal$ noetherian, it
	is a torsion class. Any object of finite length $x$ has a simple subobject
	$s$, and therefore, if it is to be torsion-simple, it must belong to
	$\serre(s)$.
\end{proof}

When passing to the limit, we lose one implication. Recall that an object $x$
admits an \emph{infinite filtration} with simple factors if there is a direct
system of monomorphisms $(x_\alpha, x_\alpha\subseteq x_\beta\mid \alpha\leq
\beta\leq \lambda)$ indexed on an infinite ordinal $\lambda$ such that $x_0=0$,
$x_\lambda=x$, the factor $x_{\alpha+1}/x_\alpha$ is simple for every
$\alpha<\lambda$, and $x_{\alpha}=\varinjlim_{\beta<\alpha}x_\beta$ for any
limit ordinal $\alpha\leq \lambda$.

\begin{prop}\label{prop:locally-finite}
	Let $\Gcal$ be a locally noetherian Grothendieck category. If a torsion-simple
	object admits a (possibly infinite) filtration with simple factors, they
	must all be isomorphic.
\end{prop}

\begin{proof}
	The proof is similar to the previous one, except that
	instead of $\serre(s)$ one considers its closure under direct limits
	$\varinjlim\serre(s)$, which is a torsion class of $\Gcal$ by work of
	Crawley-Boevey~\cite[\S4.4]{craw-94}. In fact, by closure properties, one sees that
	it is the torsion class generated by $s$ in $\Gcal$.
	Therefore, since $s$ is simple, $\varinjlim\serre(s)$ is also closed under
	subobjects, by a straightforward adaptation of \cite[Prop.~VI.3.3]{sten-75}.
	It follows that any simple factor of an object of $\varinjlim\serre(s)$ must
	belong to $\serre(s)=\varinjlim\serre(s)\cap\fp\Gcal$, and therefore be isomorphic to $s$ itself.

	Now, as in the previous proof, if $x$ admits a filtration $(x_\alpha\mid \alpha\leq\lambda)$
	with simple factors, it must belong to the torsion class
	$\varinjlim\serre(x_1)$, for which it has non-zero torsion part (containing
	$x_1\subseteq x$). It follows that all its simple factors are isomorphic to $x_1$.
\end{proof}

For an explicit counterexample to the converse implication, see
Example~\ref{ex:double-origin} and the accompanying Remark. It is worth
mentioning that the category of that counterexample is not of the form $\Mod A$
for an artinian ring $A$, leaving open the question of whether the converse of
the Proposition is true in such a category.

\begin{ex}\label{ex:path-algebra}
	Let $Q$ be a finite acyclic quiver and $k$ an algebraically
	closed field. The category $\Mod {kQ}$ is locally finite, so every module
	admits a (possibly infinite) filtration with simple factors. By
	Proposition~\ref{prop:locally-finite}, any torsion-simple module $M$ belongs to
	$\Scal:=\varinjlim\serre(S)$ for a simple module $S$; this class is a Serre
	subcategory of $\Mod kQ$ (see also the proof of
	Proposition~\ref{prop:locally-finite}). If $S$ corresponds to
	the vertex $i\in Q$, $\Scal$ consists of the modules concentrated in vertex
	$i$, when seen as representations of $Q$ over $k$-vectors spaces. By
	Proposition~\ref{prop:serre-restriction}, $M\in\Scal$ is torsion-simple in $\Mod kQ$
	if and only if so it is in $\Scal$; and in $\Scal\simeq \Mod k$ there are only
	two torsion classes, $0$ and $\Scal$. Therefore, every object of $\Scal$ is
	torsion-simple in ($\Scal$ and ) $\Mod kQ$. We conclude that the
	torsion-simple modules over $kQ$ are precisely the representations concentrated in a single vertex.
\end{ex}

\subsection{Modules over a commutative noetherian ring}\label{subsec:commutative}

We now focus on the categories $\mod R$ and $\Mod R$ for a commutative noetherian
ring $R$, first recalling some commutative algebra that comes into play.

By $\Spec R$ we denote the spectrum of prime ideals of $R$. This object has a
strong grip on $\Mod R$, through the following assignment. For an element $m$
of a module, denote by $\Ann(m)\leq R$ its annihilator ideal. Any module $M$ has a
set of \newterm{associated primes}:
\[\Ass M:=\{\,\pf\in\Spec R \mid \exists\, m\in M \text{ with }\Ann(m)=\pf\,\}.\]
This set is non-empty if $M\neq 0$, and finite if $M\in\mod R$
\cite[Propp.~VII.1.1 and~VII.1.6]{sten-75}.

Recall that a subset $\Vcal\subseteq\Spec R$ is \emph{specialisation-closed}
(\emph{sp-closed} for short) if it is an upper-set: ie.\ if $\pf\subseteq \qf$
and $\pf\in\Vcal$ imply $\qf\in\Vcal$. We denote by $V(\Pcal)$ the smallest
sp-closed subset of $\Spec R$ containing some $\Pcal\subseteq\Spec R$.

The associated primes classify all torsion pairs of $\mod R$, via a few of
remarkable results from the literature, which we collect in the following theorem:

\begin{thm}\label{thm:gabriel}
	All torsion pairs of $\mod R$ are hereditary. They is a bijection
	\[
		\{\,\text{torsion pairs }(\Tcal,\Fcal)\text{ in }\mod R\,\} \overset{1:1}\longleftrightarrow
			\{\,\text{specialisation-closed subsets }\Vcal\subseteq\Spec R\,\}
	\]
	given by the mutually inverse assignments
	\begin{align*}
		(\Tcal,\Fcal)&\mapsto V(\Ass \Tcal)\\
		\Vcal&\mapsto (\{M\mid \Ass M\subseteq \Vcal\}, \{M\mid \Ass M\subseteq
		\Vcal^c\})
	\end{align*}
	where $\Vcal^c:=\Spec R\setminus \Vcal$ denotes the complement.
\end{thm}

\begin{proof}
	The fact that all torsion pairs of $\mod R$ are hereditary is a
	consequence of a recent result of Angeleri Hügel and Hrbek
	\cite[Lemma~4.2]{ange-hrbe-17}. In fact, by the already cited work of
	Crawley-Boevey~\cite[\S4.4]{craw-94}, restriction and closure under direct
	limits give a bijection between torsion pairs of
	$\mod R$ and torsion pairs of finite type of $\Mod R$;
	by the cited result, the latter are all hereditary, and therefore so are
	their restrictions, ie.\ all the torsion pairs of $\mod R$. On this note,
	see also \cite[Thm.~2.5]{iima-mats-shim-taka-22}\footnote{We warn the reader
	that in this paper the term \emph{torsion-free class} is used for a class
	closed under subobjects and extensions, regardless of whether it belongs to
	a torsion pair. Since the context is that of a noetherian, but not necessarily
	artinian category of modules, this notion differs from ours.},
	where the same result is achieved by a different route.

	In view of this, the bijection in the statement is a refinement of a
	classical theorem of Gabriel
	(see for example \cite[Prop.~2.3]{ange-posp-stov-trli-14}, which includes
	the explicit description of the torsion-free class). This result is more
	commonly stated writing $\Supp M$ instead of $V(\Ass M)$ (including in the
	provided reference); see \cite[Prop.~10.63.6]{stacks} for the equality between these two sets.
\end{proof}

We immediately deduce a complete description of the torsion-simple objects of $\mod
R$:

\begin{cor}\label{cor:ass-singleton}
	A module $M$ in $\mod R$ is torsion-simple if and only if $\Ass M$ is a
	singleton.
\end{cor}

\begin{proof}
	$M\neq 0$ is torsion-simple if and only for every sp-closed set
	$\Vcal\subseteq\Spec R$ either $\Vcal$ or $\Vcal^c$ contains $\Ass M$.
	This is true if $\Ass M$ (which is non-empty) is a singleton,
	but false if it has more than one element: to see this, take a maximal
	element $\pf$ of $\Ass M$ and consider $\Vcal=V(\{\pf\})$.
\end{proof}

We invite the reader to compare this result with
Proposition~\ref{prop:finite-length}: it is general wisdom that associated
primes are the commutative noetherian analogue of the simple factors of the
finite-length case.

\begin{rmk}
	Note that from this characterisation it follows that the type of a
	torsion-simple module in $\mod R$ is determined by its unique associated
	prime; hence there is a bijection between the types of torsion-simple
	modules and the prime spectrum $\Spec R$.
\end{rmk}

In particular, since for $\pf\in\Spec R$ we have $\Ass(R/\pf)=\{\pf\}$
(as it is easily seen, or by \cite[Lemma~10.63.4]{stacks}), we obtain the following:
\begin{cor}\label{cor:r/p}
	For a commutative noetherian ring $R$, the cyclic module $R/\pf$
	is torsion-simple for every prime $\pf$.
\end{cor}

Later on, in \S\ref{subsec:commalg}, we are going to give another, more
elementary proof of this fact, using some commutative algebra rather than the
deep results behind Theorem~\ref{thm:gabriel}.

Having dealt with the finitely generated modules, we pass to $\Mod R$. First,
notice that Corollary~\ref{cor:restriction} applies to show that the finitely
generated modules having a unique associated prime are also torsion-simple in
$\Mod R$. This was not so clear \emph{a priori}, as we do not have any grasp on
\emph{all} torsion pairs in $\Mod R$ (see the footnote on
page~\pageref{foot:stanley}).

An example of (often infinitely generated) torsion-simple modules is the
following:

\begin{ex}\label{ex:residues}
	For every prime $\pf$, the residue field $k(\pf):=R_\pf/\pf R_\pf$ is
	torsion-simple in $\Mod R$; this follows from
	Corollary~\ref{cor:division-ring}. Notice that this does not need $R$ to be
	noetherian.
\end{ex}

We chose to highlight this example because the residue fields are important
objects from a homological point of view. On the other hand, the same Corollary
also ensures that the other division $R$-algebras (for example, the quaternions
$\Hbb$ in $\Mod \Zbb$) are torsion-simple.

Our next objects of interest are the indecomposable injective modules. By a famous
result of Matlis \cite[Thm.~2.4]{matl-58}, these are the injective envelopes $E(R/\pf)$ of the
modules $R/\pf$. $E(R/\pf)$ has $\pf$ as its unique associated prime;
this is also true for all its finitely generated submodules, which are then
torsion-simple of the same type. In other words, for any torsion pair, either
all the finitely generated submodules of $E(R/\pf)$ are torsion, or they are
torsion-free. In the first case, $E(R/\pf)$ is torsion as well, since torsion
classes are closed under direct limits (and therefore unions). In the second
case, however, it is harder to control $E(R/\pf)$. Already when $R=\Zbb$, we can
observe this behaviour:

\begin{ex}
	Consider the torsion pair $(\Dcal,\Rcal)$ with $\Dcal=\Inj \Zbb$ being the
	class of divisible groups, and $\Rcal=\Dcal^\bot$ the class of reduced groups. Notice
	that $\mod \Zbb\subseteq\Rcal$: hence all finitely generated
	submodules of the injectives (including those of $E(\Zbb/\pf)$, for every prime
	$\pf$) are torsion-free, while the injectives
	themselves are not, and are in fact torsion.
\end{ex}

\noindent This is not a counterexample to the indecomposable injectives being
torsion-simple. Indeed, this claim is true for Dedekind domains, as we are going
to show now.

Notice first that to check whether $E(R/\pf)$ is torsion-simple, we may assume
without loss of generality that $R$ is local with maximal ideal $\pf$. Indeed,
$E(R/\pf)$ has a natural structure of $R_\pf$-module, so we can apply
Corollary~\ref{cor:localisation} to the ring epimorphism $R\to R_\pf$.
If we start with a Dedekind domain $R$, passing to the localisation we get a
PID, and we can apply the following result:

\begin{lemma}\label{lemma:PID}
	Let $(R,\mf)$ be a local PID. Then the injective envelope of the residue field
	$k$ is torsion-simple.
\end{lemma}

\begin{proof}
	If $R$ is a field, the claim is trivial, as there are only the two trivial
	torsion pairs in $\Mod R$. Otherwise, we are argue towards an application of
	Proposition~\ref{prop:injective}.
	First, notice that all the finitely generated submodules of $E(k)$, having
	$\mf$ as their unique associated prime, have finite length; hence $E(k)$ is
	the union of its submodules of finite length.
	Let $0\neq m\in R$ be a generator of
	the maximal ideal, and consider the endomorphism of $E(k)$ given by
	multiplication by $m$. The kernel of this endomorphism lies in the (simple) socle of
	$E(k)$, since it is annihilated by the maximal ideal; therefore it is
	precisely the socle $k\subseteq E(k)$, which is contained in every non-zero submodule. On the
	other hand, by injectivity, we can `divide' by $m$ every element of $E(k)$
	annihilated by $\Ann(m)=0$; this means that our endomorphism is surjective.
	We conclude that there are no proper torsion parts containing $k$; hence the
	only torsion parts of $E(k)$ are $0$ and $E(k)$.
\end{proof}

As explained above, we deduce the following:

\begin{prop}\label{prop:dedekind}
	Let $R$ be a Dedekind domain. Then every indecomposable injective $R$-module
	is torsion-simple in $\Mod R$.
\end{prop}

\begin{proof}
	For every $\pf\in\Spec R$, by Corollary~\ref{cor:localisation} the
	indecomposable injective $E(R/\pf)$ is torsion-simple over $R$ if and only
	if so it is over $R_\pf$, where it is the injective envelope of the residue
	field. A localisation $R_\pf$ of a Dedekind domain is a PID (see eg.\
	\cite[\S9]{atiy-macd-69}), so we conclude by applying Lemma~\ref{lemma:PID}.
\end{proof}

Part of this Proposition, namely the fact that the injective envelopes of the
simple modules are torsion-simple, can also be deduced by the following
observation:

\begin{prop}\label{prop:simple-envelopes}
	Let $\Gcal$ be a hereditary Grothendieck category, that is, with
	$\gldim\Gcal\leq 1$, and let $s$ be a simple object. If the injective
	envelope $E(s)$ of $s$ is contained in the torsion class generated by $s$,
	it is torsion-simple.
\end{prop}

\begin{proof}
	Let $\Tcal$ be the torsion class generated by $s$, and let
	$E:=E(s)$ be the injective envelope of $s$.
	Let $0\neq M\subset E$ be a non-zero proper subobject; notice that it must
	contain $s\subseteq E$, which is simple and essential.
	Consider the quotient $E/M$, which is injective because $\idim M\leq 1$.
	Since $E$ belongs to $\Tcal$, so does $E/M$; in particular, it does not
	belong to the torsion-free class $\Tcal^\bot=s^\bot$. In other words, there
	is a non-zero (=monic) morphism $s\to E/M$. By injectivity of $E/M$, this
	lifts along the inclusion $s\subseteq M$ to a non-zero morphism $M\to E/M$.
	Hence $M$ is not a torsion part of $E$.
\end{proof}

\begin{cor}
	Let $R$ be a hereditary commutative noetherian ring, and $\mf$ a maximal
	ideal. Then the indecomposable injective $E(R/\mf)$ is torsion-simple.
\end{cor}

\begin{proof}
	The torsion pair generated by a simple object is always hereditary; and in
	$\Mod R$, the hereditary torsion classes are closed under taking injective
	envelopes (ie.\ they are \emph{stable}, see \cite[Propp.~VI.7.1
	and~VII.4.5]{sten-75}). The conclusion follows from the Proposition.
\end{proof}

\begin{rmk}
	Notice that the stability of all hereditary torsion pairs is not common outside
	the commutative setting: for example, already on the path algebra of the
	$\Abb_2$ quiver, the hereditary torsion pair generated by the simple
	projective does not contain its injective envelope (the
	projective-injective).
\end{rmk}

As mentioned, if $R$ is a Dedekind domain, it satisfies the assumptions of this
Corollary, hence we recover part of Proposition~\ref{prop:dedekind}.

We conclude with a counterexample, showing that the statement of
Proposition~\ref{prop:dedekind} is not true for every
commutative noetherian ring.

\begin{ex}\label{ex:double-origin}
	Consider the local ring $R:=(\Cbb[x,y]/(x^2,xy))_{(x,y)}$, obtained by
	localising the `line with the double origin' at the `double origin' itself.
	Denote by $\mf=(x,y)$ the maximal ideal.
	Let $E$ be the injective envelope of the residue field $\Cbb$, and let
	$T:=\Ann^E(x)$ be the submodule of $E$ consisting of the elements annihilated by
	$x$: it is non-zero, as it contains $k\subseteq E(k)$. Since $x(x,y)=0$, we have
	that $\mf E\subseteq T$; this shows that $\mf(E/T)=0$, that is, this
	quotient is a direct sum of copies of the residue field. In particular, it
	has finite length (a computation shows that it is
	actually isomorphic to $\Cbb$; see \cite[\S31]{lam-99} for useful details on the
	structure of $E$). On the other hand, since $\Ann(y)=(x)$, for every
	element $t\in T$ one can define a morphism $\mf\to T, x\mapsto 0, y\mapsto
	t$. By injectivity of $E$, this extends to a morphism $R\to E,1\mapsto e$
	with $xe=0$ and $ye=t$. In particular, $e\in\Ann^E(x)=T$; this shows
	that every element of $T$ is `divisible by $y$' in $T$, and therefore we
	have $\mf T=T$. This property must carry over to any quotient of $T$, and
	therefore $T$ does not have any finitely generated quotients by Nakayama's
	Lemma; equivalently, it does not have non-zero morphisms to finitely
	generated modules.

	As a result, the sequence $(0\to T\to E\to E/T\to 0)$ is a torsion sequence
	with respect to all torsion pairs whose torsion class lies bewteen the one
	generated by $T$ and the one cogenerated by $E/T$ (equivalently, by $\Cbb$).
\end{ex}

\begin{rmk}
	In the Example, all finitely generated subobjects of $E$ have finite length,
	as they have $\mf$ as their unique associated prime. Therefore, $E$ admits
	a(n infinite) filtration with simple factors, all isomorphic to $\Cbb$
	(the unique simple). For this reason, this is also a counterexample to the
	converse implication of Proposition~\ref{prop:locally-finite}.
\end{rmk}

We remark that the ring of this counterexample is not even Cohen-Macaulay. It
remains open to decide whether Proposition~\ref{prop:dedekind} generalises to
wider classes of rings.

\subsection{Quasi-coherent sheaves over $\Pbb^1(k)$}\label{subsec:sheaves}

In this section, we use the results obtained for Dedekind domain to address the
categories $\coh(\Pbb^1)$ and $\Qcoh(\Pbb^1)$ of (quasi)coherent sheaves over
the projective line $\Pbb^1(k)$, for an algebraically closed field $k$.

The relation with the previous section is given by the open covering of $\Pbb^1$
with two copies of the affine line $\Spec{k[t]}$, identified with the open
subsets $U_0:=\Pbb^1\setminus\{\infty\}$ and $U_\infty:=\Pbb^1\setminus \{0\}$.
The inclusions $i,j$ of these open subschemes induce adjunctions
\[\begin{tikzcd}
	\Qcoh(\Pbb^1) \arrow[shift left]{r}{i^{-1}} & \arrow[shift left]{l}{i_\ast} \Qcoh(U_0)\simeq \Mod k[t]
	\qquad
	\Qcoh(\Pbb^1) \arrow[shift left]{r}{j^{-1}} & \arrow[shift left]{l}{j_\ast} \Qcoh(U_\infty)\simeq \Mod k[t]
\end{tikzcd}\]
where $i^{-1}$ and $j^{-1}$ are the restriction functors, while $i_\ast$ and $j_\ast$
are the direct image functors defined by
\[\Gamma(U,i_\ast\Fcal)=\Gamma(U\cap U_0,\Fcal)\text{ for every }\Fcal\in\Qcoh(U_0)\]
and similarly for $j_\ast$ (see \cite[\S{}II~Ex.~1.18]{hart-77}). This
identifies $\Qcoh(U_0)$ with the full subcategory of $\Qcoh(\Pbb^1)$ consisting
of sheaves for which the unit of the adjunction $(i^{-1},i_\ast)$ is an
isomorphism, that is, those $\Fcal$ such that $\Gamma(U,\Fcal)=\Gamma(U\cap
U_0,\Fcal)$ for every open $U\subseteq \Pbb^1$; and similarly for
$\Qcoh(U_\infty)$. These functors restrict to identify $\coh(U_0)$ and
$\coh(U_\infty)$ with two subcategories of $\coh(\Pbb^1)$.
In turn, the equivalence $\Qcoh(U_0)\simeq\Mod{k[t]}$ identifies a
quasi-coherent sheaf $\Fcal$ on $U_0$ with its $k[t]$-module of (global) sections
$\Gamma(U_0,\Fcal)$, and it restricts to an equivalence
$\coh(U_0)\simeq\mod{k[t]}$.

For every point $P$, one can compute the \emph{stalk} of a sheaf $\Fcal$ at $P$,
which is an abelian group $\Fcal_P$: this defines a functor $\Fcal\mapsto
\Fcal_P$ which is exact and commutes with coproducts. Therefore we have:

\begin{lemma}
	For a set $\Scal$ of closed points of $\Pbb^1$, the class $\Zcal_{\Scal}$ of
	quasicoherent sheaves $\Fcal$ for which $\Fcal_P=0$
	for every $P\in\Scal$ is a hereditary torsion class in $\Qcoh(\Pbb^1)$.
\end{lemma}

\begin{proof}
	$\Zcal_{\Scal}$ is closed under extensions, subobjects, quotients and
	coproducts, hence it is a hereditary torsion class in $\Qcoh(\Pbb^1)$.
\end{proof}

If we choose $\Scal=\{\infty\}$, we have an inclusion
$\Zcal_{\{\infty\}}\subseteq\Qcoh(U_0)$:

\begin{lemma}
	Let $\Fcal\in\Qcoh(\Pbb^1)$ be such that the stalk at infinity is
	$\Fcal_\infty=0$. Then, for any open subset $U\subseteq\Pbb^1$, we have
	$\Gamma(U,\Fcal)\simeq\Gamma(U\cap U_0,\Fcal)$.
\end{lemma}

\begin{proof}
	There is nothing to prove if $U\subseteq U_0$, so we assume that instead
	$\infty\in U$.
	Since $U_0$ is dense in $U$, the restriction map
	$\Gamma(U,\Fcal)\to \Gamma(U\cap U_0,\Fcal)$ is monic. To show
	that it is epic, let $s$ be a section on $U\cap U_0$. Since $s$
	vanishes in the stalk at $\infty$ by assumption, there exists an open
	neighbourhood $V$ of $\infty$ such that $s_{|U\cap V}=0$. Hence we
	can glue $s$ with the zero section on $V$, thus extending $s$ to a section
	on $U$.
\end{proof}

We use this subcategory $\Zcal_{\{\infty\}}$ as a `pivot' to prove the
following characterisation of torsion-simple objects.

\begin{prop}\label{prop:sheaf-with-gap}
	Let $\Fcal\in\Qcoh(\Pbb^1)$ be such that $\Fcal_\infty=0$. Then $\Fcal$ is
	torsion-simple in $\Qcoh(\Pbb^1)$ if and only if $\Gamma(U_0,\Fcal)$ is
	torsion-simple in $\Mod{k[t]}$.
\end{prop}

\begin{proof}
	The assumption is that $\Fcal\in \Zcal_{\{\infty\}}\subseteq\Qcoh(U_0)\simeq
	\Mod{k[t]}$: we are claiming that $\Fcal$ is torsion-simple in $\Qcoh(\Pbb^1)$
	if and only if so it is in $\Qcoh(U_0)$. To this end, we would like to apply
	Proposition~\ref{prop:serre-restriction}; but we cannot do it directly, as
	$\Qcoh(U_0)$ is not a Serre subcategory of $\Qcoh(\Pbb^1)$ (it is not closed
	under subobjects). On the other hand, $\Zcal_{\{\infty\}}$ is a Serre
	subcategory of both categories of quasi-coherent sheaves over $U_0$ and
	$\Pbb^1$: we can then apply the Proposition twice. Indeed,
	$\Fcal\in\Zcal_{\{\infty\}}$ is torsion-simple in $\Qcoh(\Pbb^1)$ if and only
	if so it is in $\Zcal_{\{\infty\}}$; and it is torsion-simple in
	$\Zcal_{\{\infty\}}$ if only if so it is in $\Qcoh(U_0)$.
\end{proof}

For coherent sheaves, this yields the following partial classification.

\begin{cor}
	Let $\Fcal\in\coh(\Pbb^1)$, and assume that $\Fcal_P=0$ for some (closed)
	point $P$.
	Then $\Fcal$ is torsion-simple in $\Qcoh(\Pbb^1)$ if and only if $\{Q\mid
	\Fcal_Q\neq 0\}$ has exactly one element.
\end{cor}

\begin{proof}
	Without loss of generality (up to changing the open covering of $\Pbb^1$), we
	may assume that $\Fcal_\infty=0$. Then $\Fcal$ is torsion-simple in
	$\Qcoh(\Pbb^1)$ if and only if $\Gamma(U_0,\Fcal)\in\mod{k[t]}$ is torsion
	simple in $\Mod{k[t]}$. By the previous section, this is the case if and only
	if it has a unique associated prime. Then this prime corresponds to a closed point
	$Q$ of $\Pbb^1$, and the condition translates to $\Fcal_P=0$ for every $P\neq
	Q$.
\end{proof}

\begin{rmk}
	In this section, we are dealing with $\Pbb^1$, as we want to reduce to the
	Dedekind case to treat injectives. However, it is clear that the same result
	holds over $\Pbb^n$, for every $n$.
\end{rmk}

In the rest of this section, we address the indecomposable injectives of
$\Qcoh(\Pbb^1)$, starting with a complete list.
For every closed point $P$ of $\Pbb^1(k)$,
there is a skyscraper sheaf $\Ecal_P$ with sections
\[\Gamma(U,\Ecal_P)=\begin{cases} E(P) &\text{if }P\in U\\0
&\text{otherwise}\end{cases}\]
where $E(P)$ is the injective envelope (as a $k[t]$-module) of the residue
field of the stalk $\Ocal_P$ of the structural sheaf.
In addition to these, there is the \emph{generic} sheaf $\Gcal$, which is the constant
sheaf with sections $k(t)$ on every non-empty open.

It is clear that any constant sheaf, and therefore also $\Gcal$, belongs to
$\Qcoh(U_0)$. Moreover, if $P\neq \infty$ is a closed point, also
$\Ecal_P\in\Zcal_{\{\infty\}}\subseteq \Qcoh(U_0)$ (while similarly
$\Ecal_\infty\in\Qcoh(U_\infty)$).

This shows that for $\Ecal_P$ with $P$ a closed point we can apply
Proposition~\ref{prop:sheaf-with-gap}.

\begin{lemma}
	For every closed point $P$ of $\Pbb^1_k$, the indecomposable injective sheaf
	$\Ecal_P$ is torsion-simple in $\Qcoh(\Pbb^1)$.
\end{lemma}

\begin{proof}
	Up to changing open covering of $\Pbb^1$, we may assume that $P\neq \infty$.
	Then, by Proposition~\ref{prop:sheaf-with-gap}, the indecomposable injective
	$\Ecal_P\in\Zcal_{\{\infty\}}$ is torsion-simple in $\Qcoh(\Pbb^1)$ if and
	only if $\Gamma(U_0,\Ecal_P)\simeq E(P)$ is torsion-simple in $\Mod{k[t]}$. This module
	is an indecomposable injective over a PID, and therefore it is torsion-simple by
	Proposition~\ref{prop:dedekind}.
\end{proof}

Notice that the same argument does not apply to the generic injective $\Gcal$,
which is not contained in $\Zcal_\Scal$ for any $\Scal\neq\emptyset$ because its
stalks never vanish. We conclude by showing directly that $\Gcal$ is
torsion-simple in $\Qcoh(\Pbb^1)$ nonetheless.

\begin{prop}
	Every indecomposable injective of $\Qcoh(\Pbb^1)$ is torsion-simple.
\end{prop}

\begin{proof}
	We prove that $\Gcal$, the constant sheaf with sections $k(t)$, is
	torsion-simple. First, as we mentioned $\Gcal\in\Qcoh(U_0)\simeq
	\Mod(k[t])$, where it corresponds to the module $k(t)$. Since $\Qcoh(U_0)$ is
	a full subcategory, we have isomorphisms
	$\End_{\Qcoh(\Pbb^1)}(\Gcal)\simeq\End_{\Qcoh(U_0)}(\Gcal)\simeq
	\End_{k[t]}(\Gamma(U_0,\Gcal))\simeq k(t)$. If $\Fcal\subseteq\Gcal$
	is any subsheaf closed under endomorphisms (not necessarily lying in
	$\Qcoh(U_0)$), then $\Gamma(U_0,\Fcal)\subseteq\Gamma(U_0,\Gcal)$ is also a
	submodule closed under endomorphisms, because any endomorphisms of the section
	module of $\Gcal$ lifts to an endomorphism of $\Gcal$, by the isomorphisms
	above. Similarly, $\Gamma(U_\infty,\Fcal)\subseteq\Gamma(U_\infty,\Gcal)\simeq
	k(t)$ is closed under endomorphisms. Therefore both
	$\Gamma(U_0,\Fcal)$ and $\Gamma(U_\infty,\Fcal)$ are isomorphic either to $k(t)$ or to $0$.
	Of the four combinations of these restrictions, only two agree on the
	intersection, and they glue to the subobjects $0, \Gcal\subseteq \Gcal$.
	These are therefore the only torsion parts of $\Gcal$ in $\Qcoh(\Pbb^1)$.
\end{proof}

\begin{rmk}
	It is well known that the category $\Qcoh(\Pbb^1)$ is derived-equivalent to
	the category $\Mod k(\begin{tikzcd}[cramped,sep=small]\bullet\arrow[shift
	left]{r}\arrow[shift right]{r}&\bullet\end{tikzcd})$ of modules over the
	path algebra of the Kronecker quiver (originally proved by
	Beĭlinson~\cite{beil-78}; see Bondal~\cite[Ex.~6.3]{bond-89}).
	In fact, these two abelian categories are related by HRS-tilting at a torsion
	pair (see \cite[Cor.~I.2.2]{happ-reit-smal-96} for the definition,
	\cite[Figure~5.3]{kell-07} for a picture).
	Under this perspective, the injective sheaves $\Ecal_P$ are the Prüfer
	modules over the Kronecker, while the generic injective $\Gcal$ is the
	generic module. By Example~\ref{ex:path-algebra}, none of these modules over
	the Kronecker is torsion-simple. This shows that the torsion-simplicity of
	an object of an abelian category $\Acal$ cannot be recognised from the point
	of view of the derived category $\D(\Acal)$ of $\Acal$.
\end{rmk}

\subsection{An alternative proof of Corollary~\ref{cor:r/p}}\label{subsec:commalg}

As promised, we now present a more elementary proof of the fact that the cyclic
module $R/\pf$ is torsion-simple, for a commutative noetherian ring $R$ and a
prime ideal $\pf$.

The proof will be split into a few lemmas; but first we rephrase the claim. We
want to prove that for every non-trivial proper submodule $N$ of $R/\pf$, there
is a non-zero morphism $N\to (R/\pf)/N$: since $N$ has the form $I/\pf$ for
$\pf<I<R$, we want to prove that
\[0\neq \Hom_R(I/\pf,R/I)\simeq\Hom_{R/\pf}(I/\pf,R/I).\]
Since $S:=R/\pf$ is a domain, this argument shows that our claim is equivalent to
the following fact, which we are going to prove instead.

\begin{prop}\label{prop:morphisms}
	Let $S$ be a commutative noetherian domain and $0\neq I <S$ a nonzero proper
	ideal. Then $\Hom_S(I,S/I)\neq 0$.
\end{prop}

From now on we will adhere to the notation of this statement. We immediately
notice that $\Hom_S(I,S/I)\simeq \Hom_{S/I}(I/I^2,S/I)$ by $\otimes$-$\Hom$
adjunction. Consider a free
$S/I$-presentation of the finitely presented module $I/I^2$, with relation
matrix $M\in M_{n\times m}(S/I)$:
\[\begin{tikzcd}[column sep=1.5pc]
	(S/I)^m \arrow{r}{M} & (S/I)^n \arrow{r} & I/I^2 \arrow{r} & 0.
\end{tikzcd}\]

Applying the $S/I$-dual $\Hom_{S/I}(-,S/I)$ we obtain an exact sequence:
\[\begin{tikzcd}[column sep=1.5pc]
	0 \arrow{r} &\Hom_{S/I}(I/I^2,S/I) \arrow{r} & (S/I)^n \arrow{r}{M^t} & (S/I)^m
\end{tikzcd}\]
which identifies the first $\Hom$-space with the kernel of the transpose
$M^t\in M_{m\times n}(S/I)$. Now we recall some tools to study this kernel.

\begin{dfn}[McCoy {\cite[pg.\ 159]{mcco-62}}]
	Let $A\in M_{m\times n}(S)$ be a matrix. For $r\geq 0$ define
	$D_r(A)$ as the ideal of $S$ generated by the minors of $A$ of order $r$,
	where conventionally $D_0(A)=S$. Notice that $S=D_0(A)\supseteq
	D_1(A)\supseteq\cdots\supseteq D_{\mathsf{min}\{m,n\}}(A)\supseteq
	0=\cdots$, by the Laplace expansion formula for the determinant. The \emph{McCoy
	rank} $\mcrk A$ of $A$ is defined as the largest $r\geq 0$ for which
	$\Ann(D_r(A))=0$.
\end{dfn}

The McCoy rank works similarly to the usual rank of matrices over a field.

\begin{thm}[{\cite[Thm.~51]{mcco-62}}]
	A matrix $A\in M_{m\times n}(S)$ has nonzero nullvectors if and only if $\mcrk A<n$.
\end{thm}

Our goal now is to show that $\mcrk{M^t}=\mcrk{M}$ is less that $n$, and then use this result. For this
we need a couple of technical lemmas.

\begin{lemma}
	Let $S$ be a commutative noetherian domain, $0\neq I\leq S$ a nonzero ideal
	and $d\in S$. If $dI\subseteq I^2$, then $d\in \sqrt I$.
\end{lemma}

\begin{proof}
	Since $I$ is finitely generated, we argue by induction on the number of
	generators. If $I=(x)$ is principal and $dI\subseteq I^2$, this means that
	$dx=cx^2$, so $(d-cx)x=0$. Since $S$ is a domain and $x\neq 0$, we obtain
	$d=cx\in I\subseteq\sqrt I$.

	Assume $I$ has $n$ generators $x_1,\dots, x_n$; from now on $i,j,k$ are
	understood to be indices $1\leq i,j,k\leq n$. Since $I^2$ is generated by
	$\{x_ix_j\mid i\leq j\}$, and by assumption $dx_k\in I^2$
	for every $k$, let $c_{ij}^k\in S$ be coefficients
	such that $dx_k=\sum_{i\leq j} c_{ij}^kx_ix_j$.

	Let $d':=d-\sum_{j}c_{1j}^1x_j$: since $d-d'\in I$, we have that
	$dI\subseteq I^2$ if and only if $d'I\subseteq I^2$, and also $d\in \sqrt
	I$ if and only if $d'\in \sqrt I$.
	So, up to substituting $d:=d'$, we may assume that
	$c_{1j}^1=0$ for every $j$. In other words, we can assume that $dx_1\in
	J^2$, where $J:=(x_2,\dots,x_n)\subseteq I$. Therefore, we also obtain that
	$d^3J\subseteq d^2dI\subseteq d^2I^2=(dI)^2\subseteq J^2$.
	By the inductive hypothesis applied to $d^3$ and $J$,
	we deduce that $d^3\in \sqrt J\subseteq \sqrt I$, and so $d\in \sqrt I$.
\end{proof}

\begin{ex}
	This result fails if we drop the domain assumption: an easy counterexample
	is $S=\Cbb[x,y]/(xy)$, $d=x+y$ and $I=(x)$. Unsurprisingly, this is also a
	counterexample to our Proposition without the domain assumption, as
	$I=(x)\simeq k[x]\leq S$ has no nonzero morphism to $S/I\simeq k[y]$.
\end{ex}

\begin{lemma}
	All the minors of order $n$ of $M$ are nilpotent in $S/I$.
\end{lemma}

\begin{proof}
	Assume $M$ has a minor $\bar d = d+I \in S/I$ of order $n$ which is not
	nilpotent. Then, we can localise $S/I$ away from $\bar d$, inverting the multiplicative set
	$\{1,\bar d,\bar d^2,\dots\}$. The $n\times n$-submatrix of $M$ with
	determinant $\bar d$ then becomes
	invertible, which makes $M$ have a right-inverse; hence, the
	localisation of the map represented by $M$ is surjective. Since localisation
	is right-exact, we deduce that the localisation of $I/I^2$ away from $\bar d$
	vanishes. This is the case if and only if every element of $I/I^2$ is
	annihilated by a power of $\bar d$, and since $I/I^2$ is finitely generated,
	there is a power $\bar d^r$ of $\bar d$ which annihilates the whole module.
	This means that $d^rI\subseteq I^2$, and by the previous lemma we deduce
	that $d\in \sqrt I$. Therefore, $\bar d$ is nilpotent, which is a
	contradiction.
\end{proof}

\begin{cor}
	$\mcrk M<n$.
\end{cor}

\begin{proof}
	By the lemma above, all the $n$-minors of $M$ are nilpotent, ie.\ they are
	contained in the nilradical $N\leq S/I$. Since $S/I$ is noetherian, $N^r=0$
	for some minimal $r\geq1$, so we have that $D_n(M)\subseteq N$ is annihilated by
	$N^{r-1}\neq 0$ (where $N^0$ should be read as $S$, in case $S$ is reduced and
	$r=1$).
\end{proof}

By McCoy's~Theorem, this proves Proposition~\ref{prop:morphisms}, which as
observed at the beginning is equivalent to our claim.

\begin{rmk}\label{rmk:noetherian}
	As a final observation, notice that we only used the hypothesis that the ring
	$R$ is noetherian to ensure that $S:=R/\pf$ is as well. Therefore the proof
	can be applied also when $R$ is not noetherian, but $R/\pf$ is.
\end{rmk}

\bibliographystyle{plain}
\bibliography{torsion-simple.bib}
\end{document}